\documentclass[reqno]{amsart}

\usepackage{amsmath,amssymb,amsthm}

\usepackage{tikz}

\usepackage{caption}
\usepackage{graphicx}
\usepackage{graphics}

\newtheorem*{thm}{Theorem}
\newtheorem{lemma}{Lemma}
\newtheorem{corollary}{Corollary}

\DeclareMathOperator{\tr}{tr}

\begin{document}

\title[]{Randomly Pivoted Partial Cholesky: random how?}

\author[]{Stefan Steinerberger}
\address{Department of Mathematics, University of Washington, Seattle, WA 98195, USA}
\email{steinerb@uw.edu}

\begin{abstract}  We consider the problem of finding good low rank approximations of symmetric, positive-definite $A \in \mathbb{R}^{n \times n}$.  Chen-Epperly-Tropp-Webber showed, among many other things, that the randomly pivoted partial Cholesky algorithm that chooses the $i-$th row with probability proportional to the diagonal entry $A_{ii}$ leads to a universal contraction of the trace norm (the Schatten 1-norm) in expectation for each step. We show that if one chooses the $i-$th row with likelihood proportional to $A_{ii}^2$ one obtains the same result in the Frobenius norm (the Schatten 2-norm). Implications for the greedy pivoting rule and pivot selection strategies are discussed.

\end{abstract}

\thanks{The author was partially supported by the NSF (DMS-212322).}

\maketitle

\section{Introduction}
\subsection{Randomly pivoted Cholesky.} We consider the problem of finding a rank $k$ approximation of a symmetric, positive-definite matrix $A \in \mathbb{R}^{n \times n}$ in the setting where $n$ is large
and evaluating entries of the matrix is expensive. A useful method in this setting is
the randomly pivoted partial Cholesky decomposition which provides a rank-$k$ approximation using only $(k+1)n$ entry evaluations and $\mathcal{O}(k^2 n)$ additional arithmetic operations.
 Our presentation is inspired by the extensive theoretical and numerical analysis undertaken by Chen-Epperly-Tropp-Webber \cite{chen}.
The essence of the method can be described very concisely: given a symmetric, positive-definite matrix $M$ with rows $M_1, \dots, M_n \in \mathbb{R}^n$, we set $\widehat{M}^{(0)} = \mathbf{0} \in \mathbb{R}^{n \times n}$, initialize $M^{(0)} = M$ and iteratively update both matrices with a rank-one correction induced by a row $M_{s_k}$ where $s_k \in \left\{1,2,\dots, n\right\}$,
\begin{align*}
\widehat{M}^{(k+1)} &= \widehat{M}^{(k)} + \frac{ (M^{(k)}_{s_k})^{T}  (M^{(k)}_{s_k})^{}}{M^{(k)}_{s_k s_k}} \\
M^{(k+1)} &= M^{(k)} - \frac{ (M^{(k)}_{s_k})^{T}  (M^{(k)}_{s_k})^{}}{M^{(k)}_{s_k s_k}}.
\end{align*}
By construction, for all $k \geq 0$,
$$ \widehat{M}^{(k)} +M^{(k)} = M$$
and $\widehat{M}^{(k)}$ is a rank $k$ approximation of $A$ with $M^{(k)}$ being the residual. The algorithm is easy to implement and easy to run. The remaining question is how to choose the rows $s_k \in \left\{1,2,\dots, n\right\}$. Ideally, we would like to choose the rows in such a way that they capture a lot of the $\ell^2-$energy of the matrix at each step: however, we are also trying to avoid evaluating entries of the matrix beyond what is absolutely necessary and will usually only have access to the diagonal entries of $M^{(k)}$. At this point, different philosophies start to emerge.
 \begin{enumerate}
\item \textit{Random Pivoting.} Pick $s_k \in \left\{1,2,\dots, n\right\}$ at random, e.g. \cite{williams}.
\item \textit{Greedy}. Choose $s_k$ so as to select the largest diagonal entry
$$ s_k = \arg\max_{1 \leq i \leq n} M^{(k)}(i,i).$$
If $X \in \mathbb{R}^{n \times n}$ is spd, then the inequality
$ |X_{ij}| \leq \sqrt{X_{ii} X_{jj}},$
shows that small diagonal elements imply that the entire column is small. The converse is not necessarily true but in the absence of other information, why not.
\item \textit{Adaptive Random Pivoting.} Chen-Epperly-Tropp-Webber \cite{chen} propose to select the  $s_k \in \left\{1,2,\dots, n\right\}$ with a probability proportional to the size of the diagonal entry
\vspace{-3pt}
$$ \mathbb{P}\left(s_k = i\right) = \frac{ M^{(k)}(i,i)}{ \tr M^{(k)}}.$$
This prefers removing columns with large diagonals but is more flexible.
\item \textit{Gibbs sampling.} As also pointed out by Chen-Epperly-Tropp-Webber \cite{chen}, one may want to introduce a parameter $\beta \geq 0$ and set
$$ \mathbb{P}\left(s_k = i\right) \qquad \mbox{to be proportional to} \qquad  M^{(k)}(i,i)^{\beta}.$$
This contains random pivoting ($\beta = 0$), adaptive random pivoting ($\beta = 1$) and greedy pivoting ($\beta \rightarrow \infty$). We will consider $\beta=2$.
\end{enumerate}

\subsection{The case $\beta = 1$}

A big advantage of adaptive random pivoting, $\beta = 1$, is the beautiful arising mathematical structure. If we are given an spd matrix $A \in \mathbb{R}^{n \times n}$ and produce $B\in \mathbb{R}^{n \times n}$ using adaptive random pivoting, one obtains in expectation
$$  \mathbb{E}B = A - \frac{A^2}{\tr A}.$$
This leads to a rapid decay of large eigenvalues of $A$ which is an excellent property when trying to obtain good low-rank approximations. The map $\Phi(A) = \mathbb{E}A$ has a number of other desirable and useful properties \cite[Lemma 5.3]{chen}. Taking the trace of the expectation, we have 
$$\mathbb{E} \tr( B) = \left( 1 - \frac{\tr(A^2)}{\tr(A)^2}\right) \tr(A)$$
which, using the Cauchy-Schwarz inequality,
$$ \tr(A^2) = \sum_{i=1}^{n} \lambda_i(A)^2 \geq \frac{1}{n} \left(\sum_{i=1}^{n} \lambda_i(A)\right)^2 = \frac{1}{n} \tr(A)^2$$
then implies
$$\mathbb{E} \tr( B) \leq \left( 1 - \frac{1}{n}\right) \tr(A).$$
The situation is even better than that: the Cauchy-Schwarz inequality is only sharp when the eigenvalues are roughly comparable (and in that case, efficient low rank approximation starts being impossible and it does not matter very much which row one picks). If some eigenvalues are bigger than others, the argument above shows the presence of a nonlinear feedback loop that leads to much better results. 
 Chen-Epperly-Tropp-Webber \cite{chen} use this in conjunction with a sophisticated argument across many iterations to prove that adaptive random pivoting can be used to obtain high quality low-rank approximations within a short number of iterations when the error is measured in the trace norm.

\section{Results}
\subsection{Frobenius norm.}
We saw above that if $B \in \mathbb{R}^{n \times n}$ arises from $A \in \mathbb{R}^{n \times n}$ by removing the $i-$th row/column with likelihood proportional to $A_{ii}$, then 
$$\mathbb{E} \tr(B) = \left( 1 - \frac{\tr(A^2)}{\tr(A)^2}\right) \tr(A)\leq \left( 1 - \frac{1}{n}\right) \tr(A).$$
This natural inequality implies rapid decay of the trace norm (Schatten 1-norm) under random adaptive pivoting. We now prove an analogous result for the Frobenius norm (Schatten 2-norm) when sampling with likelihood proportional to $A_{ii}^2$.

\begin{thm} Suppose $A \in \mathbb{R}^{n \times n}$ is spd and $B$ arises from $A$ by selecting the $i-$th row/column as pivot with likelihood proportional to $A_{ii}^2$. Then
$$ \mathbb{E} \| B\|_F^2 \leq \|A\|_F^2 - \frac{1}{\sum_{i=1}^{n} A_{ii}^2} \sum_{i=1}^{n} \|A_{i}\|_{\ell^2}^{4} \leq \left( 1 - \frac{1}{n} \right) \|A\|_{F}^2$$
\end{thm}
The first inequality is sharp when $A$ is a diagonal matrix. In contrast to the statement for the trace which is an identity, the nonlinearity complicates things and both inequalities are typically strict. Moreover, as above, we observe that the inequalities get strictly stronger in the presence of fluctuations in $\|A_i\|_{\ell^2}$. The proof implies a stronger statement, however, since that is slightly more difficult to parse we postpone its discussion to the last section.

\subsection{Greedy Pivoting}
As a byproduct, the proof of the Theorem has implications for understanding greedy pivoting ($\beta \rightarrow \infty$) and, in particular, when it may be useful. The story can be concisely summarized in the following two bounds.
 
 \begin{corollary}
 Suppose $A \in \mathbb{R}^{n \times n}$ is symmetric and positive semi-definite. If $B$ arises from removing the $i-$th row/column, then
 $$ \left\|  B  \right\|_F^2 \leq \|A\|_F^2  - \frac{\|A_i\|_{\ell^2}^4}{A_{ii}^2}.$$
 In particular, since $\|A_i\|_{\ell^2}^2 \geq A_{ii}^2$, this implies
  $$ \left\|  B  \right\|_F^2 \leq \|A\|_F^2  - A_{ii}^2.$$
 \end{corollary}
Both inequalities are sharp for diagonal matrices. The second inequality suggests that one should maximize $A_{ii}^2$ since that leads to guaranteed decay.
 However, this is only a good idea if $A_{ii}$ being large is indicative of $\|A_i\|_{\ell^2}^4/A_{ii}^2$ being large or, phrased differently, if $A_{ii}$ being large implies that the entire row/column is large. This leads to an (impractical) deterministic method of purely theoretical interest.
 
 \begin{corollary}
  Suppose $A \in \mathbb{R}^{n \times n}$ is symmetric and positive semi-definite. If $B$ arises from removing the $i-$th row/column, where
  $$ i = \arg\max_{1 \leq j \leq n \atop A_{jj} \neq 0}  \frac{\|A_j\|_{\ell^2}^2}{A_{jj}}, \quad \mbox{then} \qquad \left\|  B  \right\|_F^2 \leq \left(1- \frac{1}{n} \right) \|A\|_F^2.$$
 \end{corollary}
 
 We emphasize again that this procedure is not practical: the whole purpose of fast algorithms of this type is that they require relatively few read outs of matrix entries: computing the $\ell^2-$norm of all the columns is too costly. However, we believe it to be a valuable theoretical insight suggesting that, depending on the matrix structure, greedy pivoting may be safe in practice.

\subsection{Some examples.}  It is a natural question whether the results presented in the previous section have any practical relevance and whether they suggest any type of pivoting rule. As a starting point, we note that the contraction property in the Frobenius norm indicates that selecting pivots with likelihood proportional to $A_{ii}^2$ is likely to produce reasonable results in the Frobenius norm and therefore also in the trace norm (which is controlled by the Frobenius norm). Nonetheless, it is not a priori clear how this will manifest in concrete examples.\\

\textit{1. Diagonal Matrices.} We discuss three relevant examples.  The first example is very simple: consider the case when $A \in \mathbb{R}^{n \times n}$ is a diagonal matrix. The expected decay (of the trace norm) in each step is given by
$$ \frac{\sum_{i=1}^{n} A_{ii}^{\beta} A_{ii}}{ \sum_{i=1}^{n} A_{ii}^{\beta}} \qquad \mbox{which is monotonically increasing in}~\beta.$$
Larger values of $\beta$ lead to better results, greedy pivoting is the best. There is an interesting question: the example of diagonal matrices shows that if one wants to sample with likelihood $A_{ii}^{\beta}$ and desires to have rapid decay of the Schatten $p-$norm 
$ \|A\|^p_{p} =  \sum_{i=1}^{n} \lambda_i(A)^p,$
then a necessary condition is $\beta \geq p$. Is it also sufficient?  It is for $\beta=1$ (by Chen-Epperly-Tropp-Webber \cite{chen}) and $\beta=2$ (by our Theorem).\\

\textit{2. Random Matrices.} We consider $100 \times 100$ random spd matrices of the form $A = Q^T D Q$, where $D$ is diagonal, $D_{ii} = f(i)$ and $Q$ is a random orthogonal matrix. We measure the size of $M^{(50)}$ relative to the size of $M^{(0)}=A$ in three norms $\|\cdot \|_X$: the operator norm, the Frobenius norm and the trace norm. Due to concentration effects, the results of the experiment do not seem to depend strongly on the choice of the random orthogonal matrix $Q$ and are thus reproducible. Greedy pivoting is uniformly better but the difference is not large. Carrying out more experiments shows that these types of matrices seem to behave a bit like diagonal matrices insofar as larger values of $\beta$ seem to be slightly better.

\begin{figure}[h!]
    \centering
    \begin{minipage}{.49\textwidth}
  \begin{tabular}{ | l | c | c | c |}
    \hline
     $f(i)$  & $\|\cdot\|_{\tiny \mbox{op}}$  &  $\|\cdot\|_{F}$   &  $\|\cdot\|_{1}$ \\ \hline
    $1 + i/100$ & 0.92 & 0.68 & 0.49   \\ \hline
    $i$ & 0.82  & 0.56  & 0.40\\ \hline
     $i^3$ & 0.46  & 0.27 & 0.18 \\    \hline
          $i^5$ & 0.20 & 0.11 & 0.07 \\    \hline
  \end{tabular}
    \captionsetup{width=.8\linewidth}
    \caption{Size of $M^{(50)}$ relative to $M^{(0)} = A$ when $\beta=1$.}
    \end{minipage}%
    \begin{minipage}{0.49\textwidth}

  \begin{tabular}{ | l | c | c | c |}
    \hline
     $f(i)$  & $\|\cdot\|_{\tiny \mbox{op}}$    &  $\|\cdot\|_{F}$&  $\|\cdot\|_{1}$ \\ \hline
    $1 + i/100$ & 0.90 & 0.67 & 0.48    \\ \hline
    $i$ & 0.77 & 0.53 & 0.37 \\ \hline
     $i^3$ & 0.35  & 0.22 & 0.15 \\    \hline
          $i^5$ & 0.13  & 0.07 & 0.04 \\    \hline
  \end{tabular}
      \captionsetup{width=.8\linewidth}
    \caption{Size of $M^{(50)}$ relative to $M^{(0)} = A$ when $\beta=\infty$.}
    \end{minipage}
\end{figure}

 \textit{3. Spiral Kernel.} So far, we have seen two examples where larger values of $\beta$ lead to better results.  To illustrate a setting where this is not the case, we use a type of example that was originally used by Chen-Epperly-Tropp-Webber \cite{chen} to illustrate the failure of greedy pivoting; the example is inspired by the type of kernel matrices that may arise in machine learning. Consider the curve $\gamma(t) = (e^t \cos{t}, e^t \sin{t})$ 
 and sample 500 points unevenly from the interval $[0,64]$ to create two clusters (see Figure 4 below). The matrix is  
  $ A_{ij} = \exp\left( - \|x_i - x_j\|^2/1000 \right)$
 and our goal will be to obtain a low-rank approximation of $A$ (using, say, 50 steps of random Cholesky). 
This type of structure may cause problems: the entries on the diagonal are, initially, $A_{ii} = 1$ and are in no relation to the size of the associate columns.
 
\begin{figure}[h!]
    \centering
    \begin{minipage}{.49\textwidth}
    \centering
\includegraphics[width=0.7\textwidth]{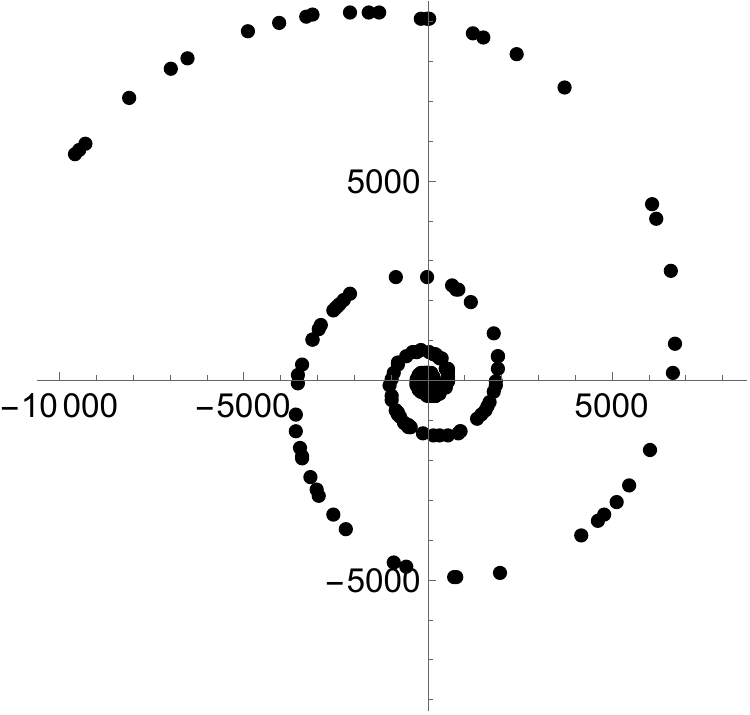}
    \captionsetup{width=.8\linewidth}
    \caption{Points on a spiral.}
    \end{minipage}%
    \begin{minipage}{0.49\textwidth}
    \centering
\includegraphics[width=0.7\textwidth]{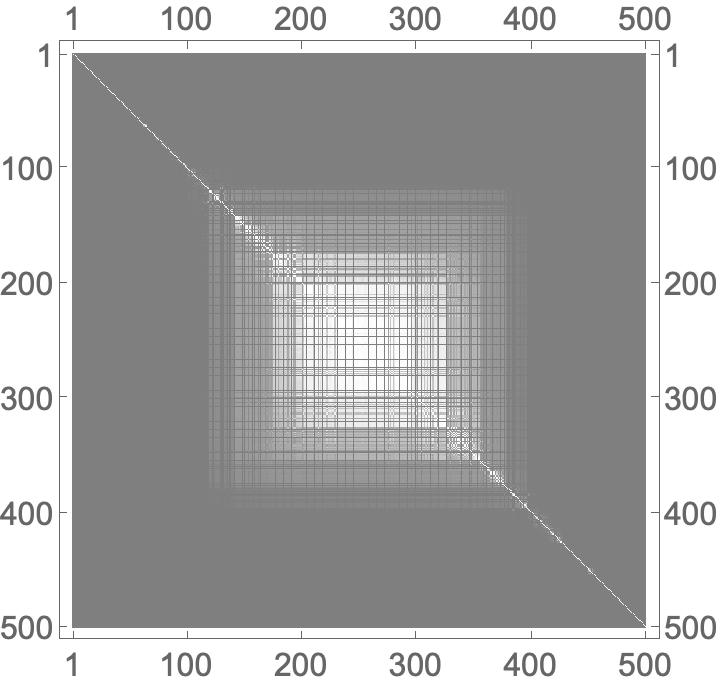}
    \captionsetup{width=.8\linewidth}
    \caption{The matrix $A$.}    \end{minipage}
\end{figure}

Choosing indices uniformly at random ($\beta = 0$) gives the best result.  Adaptive Random Pivoting ($\beta =1$) is almost as good, the method suggested by our Theorem ($\beta=2$) is virtually indistinguishable. The Greedy Method fails in a somewhat unfair way: initially, all entries on the diagonal are 1. In the implementation of the author, if the list of indices $i$ with $A_{ii}$ being maximal contains more than one element, the largest index is chosen. This means that the greedy method will erase the last row and column. This has no impact on the rest of the matrix and it will then, by induction, erase the penultimate row/column. It ends up selecting and erasing the 50 last rows/columns. A fairer implementation would be to take a random element among those that are maximal, something that is essentially done by setting, say, $\beta = 20$. $\beta = 20$ is worse than the other methods but not by much. 
  \begin{table}[h!]
  \begin{tabular}{ | l | c | c | c |}
    \hline
     Method & $\|\cdot\|_{\tiny \mbox{op}}$   &  $\|\cdot\|_{F}$ &  $\|\cdot\|_{1}$  \\ \hline
    Purely Random ($\beta = 0$) & 0.06  & 0.17 & 0.56   \\ \hline
   Adaptive Random ($\beta = 1$) & 0.09 & 0.20 & 0.57  \\ \hline
    Frobenius Random ($\beta = 2$) & 0.10  & 0.20  & 0.58\\ \hline
    Greedy ($\beta=\infty$) & 1  & 0.99  & 0.9\\ \hline
     Fixed Greedy ($\beta=20$) & 0.18  & 0.28 & 0.62 \\ \hline
        Alternating $\beta=0, \beta =\infty$ & 0.09    & 0.21 & 0.59 \\ \hline
  \end{tabular}
    \captionsetup{width=.8\linewidth}
    \caption{Ratio $\|M^{(50)}\|_X/\|M^{(0)}\|_X$ in various norms.}
\end{table}

We also note that, frequently, the performance appears to be ordered: when $\beta=1$ does better than $\beta=2$, then usually $\beta = 0$ will be better than both. Conversely, if $\beta = 2$ is better than $\beta=1$, then greedy matching $\beta = \infty$ is often better still. This naturally suggests the rule outlined in the subsequent section.

\subsection{Alternating Pivot Selection.}
There are essentially two cases: \textit{either} the size of $\|A_i\|_{\ell^2}$ is correlated with the size of $A_{ii}$, in that case we want to take greedy pivoting $\beta = \infty$ (or at the very least $\beta$ large), \textit{or} the size of $A_{ii}$ is misleading when it comes to the size of $\|A_i\|_{\ell^2}$: then it is safest to use random sampling $\beta=0$. Of course, in
practice, we do not know which of these two situations we face and, trying to avoid looking up more matrix entries, there is no way for us to find out. This suggests a natural pivoting strategy that hedges its bets.

\begin{quote}
\textbf{Alternating Pivoting Rule.} Alternate between greedy pivot selection ($\beta = \infty$) and a fully randomized pivot selection ($\beta = 0$).
\end{quote}

 We see in Table 1 that it does reasonable well in the example of points on a spiral: it is comparable to $\beta=1$ and only slightly worse than fully random selection $\beta=0$ (which, for that example, was optimal). To illustrate the strategy in another setting, we return to random matrices $A = Q^T D Q \in \mathbb{R}^{100 \times 100}$ with $Q$ being a random orthogonal matrix and the diagonal matrix being defined by $D_{ii} = f(i)$. This was the setting where greedy pivoting, $\beta = \infty$, was very good. We consider the cases, $f(i) = 1/i$ and $f(i) = i^2$ and try to recover a rank 20 approximation.  Alternating pivoting inherits the good rate from $\beta = \infty$ while also being protected against misleading diagonal entries by sampling randomly half the time.
 
\begin{figure}[h!]
    \centering
    \begin{minipage}{.49\textwidth}
  \begin{tabular}{ | l | c | c | c |}
    \hline
  & $\|\cdot\|_{\tiny \mbox{op}}$  &  $\|\cdot\|_{F}$ &  $\|\cdot\|_{1}$   \\ \hline
$\beta=0$ & 0.20 & 0.31& 0.49    \\ \hline
$\beta=1$ & 0.19 & 0.31 & 0.48  \\ \hline
$\beta=2$ & 0.18 & 0.30 & 0.48  \\    \hline
$\beta=\infty$ & 0.11 & 0.25  & 0.43 \\    \hline
 $\beta=0,\infty$ & 0.14 & 0.27 & 0.45  \\    \hline
  \end{tabular}
    \captionsetup{width=.8\linewidth}
    \caption{Size of $M^{(20)}$ relative to $M^{(0)} = A$ when $f(i) = 1/i$.}
    \end{minipage}%
    \begin{minipage}{0.49\textwidth}

  \begin{tabular}{ | l | c | c | c |}
    \hline
  & $\|\cdot\|_{\tiny \mbox{op}}$  &  $\|\cdot\|_{F}$ &  $\|\cdot\|_{1}$   \\ \hline
$\beta=0$ & 0.41 & 0.31 & 0.24   \\ \hline
$\beta=1$ & 0.37 & 0.26 & 0.21 \\ \hline
$\beta=2$ & 0.34& 0.24 & 0.20  \\    \hline
$\beta=\infty$ & 0.23  & 0.17& 0.16 \\    \hline
 $\beta=0,\infty$ & 0.29 & 0.22  & 0.20 \\    \hline
  \end{tabular}
    \captionsetup{width=.8\linewidth}
    \caption{Size of $M^{(20)}$ relative to $M^{(0)} = A$ when $f(i) = i^3$.}
    \end{minipage}
\end{figure}

 \subsection{Related results}
 This algorithm has a nontrivial history, we refer to \cite[Chapter 3]{chen} for a detailed explanation. Random pivoting for QR was considered by Frieze-Kannan-Vempala \cite{frieze} with follow-up results due to Deshpande-Vempala \cite{deshpande} and Deshpande-Rademacher-Vempala-Wang \cite{deshpande2}. Random pivoting for Cholesky does not seem to appear in the literature before being mentioned in a 2017 paper by Musco-Woodruff \cite{musco} and a 2020 paper of Poulson \cite{poulson}. These two references do not document numerical experiments and also do not investigate theoretical properties; it appears the first thorough analysis appears to be in a fairly recent paper by Chen-Epperly-Tropp-Webber \cite{chen}. Their paper also illustrates the importance of algorithms of this type when working with kernel methods in machine learning.

\section{Proofs}

\subsection{Some preparatory computations.} We start with an explicit computation that is independent of the probabilities that are chosen.
\begin{lemma}
Let $A \in \mathbb{R}^{n \times n}$ be symmetric with rows/columns $a_1, \dots, a_n$.  If we set 
$ B = A - A_i^{T} A_i/a_{ii}$ with likelihood proportional to $p_i$, then
$$ \mathbb{E} \|B\|_F^2 = \|A\|_F^2 + \sum_{i=1}^{n} p_i \left(   \frac{\|A_i\|^4}{A_{ii}^2} - 2\frac{(A^3)_{ii}}{A_{ii}}  \right).$$
\end{lemma}
\begin{proof}
We will work with the Frobenius norm in the form
$$ \int_{\mathbb{S}^{n-1}} \|Ax\|^2 dx =  \int_{\mathbb{S}^{n-1}} \left\langle Ax, Ax \right\rangle dx =  \int_{\mathbb{S}^{n-1}} \left\langle x, A^{T} Ax \right\rangle dx = c_n \|A\|_F^2,$$
where the last equation uses spherical symmetry to deduce that off-diagonal terms in $\left\langle x, A^{T} Ax \right\rangle$ average out to 0.  The constant $c_n$ depends only on the dimension and will not be important.
We start with a computation that is valid for any fixed vector $v \in \mathbb{R}^n$
\begin{align*}
\mathbb{E} \|Bv\|^2 &=  \mathbb{E} \left\| \left(  A - \frac{A_i^{T} A_i}{A_{ii}}\right)v \right\|^2 = \sum_{i=1}^{n} p_i  \left\|A v - \frac{A_i^{T} A_i}{A_{ii}}v \right\|^2 \\
&=  \|A v\|^2 -  2 \sum_{i=1}^{n} p_i \left\langle A v, \frac{A_i^{T} A_i}{A_{ii}}v \right\rangle + \sum_{i=1}^{n} p_i\left\| \frac{A_i^{T} A_i}{A_{ii}}v \right\|^2 \\
&= \|A v\|^2 + \sum_{i=1}^{n} p_i \left( \left\| \frac{A_i^{T} A_i}{A_{ii}}v \right\|^2 - 2 \left\langle A v, \frac{A_i^{T} A_i}{A_{ii}}v \right\rangle \right).
\end{align*}
Integrating $v$ on both sides of the equation over the sphere $\mathbb{S}^{n-1}$, 
$$ c_n\|B\|_F^2 = c_n\|A\|_F^2 + \sum_{i=1}^{n} p_i \cdot  \left( c_n\left\| \frac{A_i^{T} A_i}{A_{ii}} \right\|_F^2 - 2 \int_{\mathbb{S}^{n-1}} \left\langle A v, \frac{A_i^{T} A_i}{A_{ii}}v \right\rangle dv \right).$$
A computation shows
$$  \left\| \frac{A_i^{T} A_i}{A_{ii}} \right\|_F^2 = \frac{ \left\| A_i^{T} A_i \right\|_F^2}{A_{ii}^2} = \frac{\|A_i\|^4}{A_{ii}^2}$$
Using the symmetry of $A$, we have
$$2 \int_{\mathbb{S}^{n-1}} \left\langle A v, \frac{A_i^{T} A_i}{A_{ii}}v \right\rangle dv = 
2 \int_{\mathbb{S}^{n-1}} \left\langle v, A^T \frac{A_i^{T} A_i}{A_{ii}}v \right\rangle dv = 2c_n \tr\left( A^T \frac{A_i^{T} A_i}{A_{ii}} \right).$$
It remains to understand the trace of the matrix. One has
\begin{align*}
\left(A \frac{A_i^{T} A_i}{A_{ii}}\right)_{jj} &= \frac{1}{A_{ii}} \sum_{\ell=1}^{n} A_{j \ell} \left( A_i^{T} A_i \right)_{\ell j} = \frac{1}{A_{ii}} \sum_{\ell=1}^{n} A_{j \ell}  A_{i \ell} A_{i j} 
\end{align*}
Summing over $j$, we get
\begin{align*}
 \sum_{j=1}^n \left(A \frac{A_i^{T} A_i}{a_{ii}}\right)_{jj}  &=   \frac{1}{A_{ii}} \sum_{\ell=1}^{n}A_{ i \ell}  \sum_{j=1}^{n}A_{ij} A_{j \ell} = \frac{1}{A_{ii}} \sum_{\ell=1}^{n}A_{\ell i}  (A^2)_{i \ell} \\
&= \frac{1}{A_{ii}} \sum_{\ell=1}^{n}A_{\ell i}  (A^2)_{i \ell} = \frac{(A^3)_{ii}}{A_{ii}}.
\end{align*}
Altogether, we arrive at
$$ \|B\|_F^2 = \|A\|_F^2 + \sum_{i=1}^{n} p_i   \left(  \frac{\|A_i\|^4}{A_{ii}^2} - 2\frac{(A^3)_{ii}}{A_{ii}} \right).$$
\end{proof}

\subsection{An Inequality.} The previous Lemma leads to a curious quantity. One is naturally inclined to believe that this quantity should be negative since, otherwise, this would be indicative of the possibility of `bad' choices that can increase the Frobenius norm; this is indeed the case.
\begin{lemma}
If $A \in \mathbb{R}^{n \times n}$ is symmetric, positive semi-definite, then, for $1 \leq i \leq n$, 
$$ A_{ii} \cdot (A^3)_{ii} \geq \|A_i\|_{\ell^2}^4.$$
\end{lemma}
\begin{proof}
We have $ A = Q^T D Q$
   with $Q$ orthogonal and $D$ diagonal and $D_{ii} \geq 0$. Then 
   $$ (DQ)_{ij} = \sum_{k=1}^n d_{ik} Q_{kj} = d_{ii} Q_{ij}.$$
  Therefore, we can write an arbitrary entry of $A$ as
   $$ A_{ij} = \sum_{k=1}^n (Q^T)_{ik} (DQ)_{kj} =  \sum_{k=1}^n (Q^T)_{ik} d_{kk}Q_{kj} =  \sum_{k=1}^n d_{kk} Q_{ki} Q_{kj}$$
   In particular, the diagonal is positive since 
   $$ A_{ii} =  \sum_{k=1}^n d_{kk} Q_{ki}^2 \geq 0.$$
 Using the same computation together with $A^k = Q^T D^k Q$, one sees
 $$  (A^2)_{ii} = (Q^T D^2 Q)_{ii} = \sum_{k=1}^n d_{kk}^2 Q_{ki}^2$$
 as well as
$$  (A^3)_{ii} = (Q^T D^3 Q)_{ii} = \sum_{k=1}^n d_{kk}^3 Q_{ki}^2$$
Rewriting the norm of a row of $A$ as a diagonal entry of $A^2$
\begin{align*}
\|A_i\|_{\ell^2}^4 = \left( \sum_{k=1}^n A_{ik}^2\right)^2 =\left( (A^2)_{ii} \right)^2= \left( \sum_{k=1}^n d_{kk}^2 Q_{ki}^2 \right)^2.
 \end{align*}
The statement then follows from the Cauchy-Schwarz inequality
 \begin{align*}
 \|A_i\|_{\ell^2}^4 &= \left( \sum_{k=1}^n d_{kk}^2 Q_{ki}^2 \right)^2 = \left( \sum_{k=1}^n d_{kk}^{1/2} |Q_{ki}| \cdot d_{kk}^{3/2} |Q_{ki}| \right)^2 \\
 &\leq  \left(\sum_{k} d_{kk} Q_{ki}^2\right) \left( \sum_{k} d_{kk}^3 Q_{ki}^2\right) = A_{ii} \cdot (A^3)_{ii}.
 \end{align*}
   \end{proof}

\subsection{Proof of the Theorem.}
\begin{proof} The proof of Theorem 1 is now immediate. We choose the $i-$th row with probability
proportional to $A_{ii}^2$. Using Lemma 1 and Lemma 2
\begin{align*}
\mathbb{E} \|B\|_F^2 &= \|A\|_F^2 + \sum_{i=1}^{n} \frac{A_{ii}^2}{\sum_{\ell=1}^{n} A_{\ell \ell}^2} \left(   \frac{\|A_i\|^4}{A_{ii}^2} - 2\frac{(A^3)_{ii}}{A_{ii}}  \right) \\
&= \|A\|_F^2 +  \frac{1}{\sum_{\ell=1}^{n} A_{\ell \ell}^2} \sum_{i=1}^{n}\left(   \|A_i\|^4 -  2A_{ii} \cdot (A^3)_{ii} \right)\\
&\leq \|A\|_F^2 -  \frac{1}{\sum_{\ell=1}^{n} A_{\ell \ell}^2} \sum_{i=1}^{n}   A_{ii} \cdot (A^3)_{ii} \\
&\leq \|A\|_F^2 -  \frac{1}{\sum_{\ell=1}^{n} A_{\ell \ell}^2} \sum_{i=1}^{n}   \|A_i\|_{\ell^2}^4.
\end{align*}
The second part of the inequality follows from
$$ \|A\|_F^4 = \left( \sum_{i=1}^{n} \|A_i\|_{\ell^2}^2 \right)^2 \leq  n \sum_{i=1}^{n} \|A_i\|_{\ell^2}^4$$
together with $\|A_i\|_{\ell^2}^2 \geq A_{ii}^2$.
\end{proof}

\textbf{Remark.} The proof shows the stronger intermediate result
$$ \mathbb{E} \|B\|_F^2 \leq \|A\|_F^2 -  \frac{1}{\sum_{\ell=1}^{n} A_{\ell \ell}^2} \sum_{i=1}^{n}   A_{ii} \cdot (A^3)_{ii}.
$$
 Note that $A^3$ is also spd and has non-negative entries on the diagonal. Moreover,
 $$ \sum_{i=1}^{n} A_{ii} = \sum_{i=1}^{n} \lambda_i(A) \qquad \mbox{as well as} \qquad  \sum_{i=1}^{n} (A^3)_{ii} = \sum_{i=1}^{n} \lambda_i(A)^3.$$
 This leads us to the classical paradigm already hinted at above: finding a good low-rank approximation is a problem that is only relevant in the presence of eigenvalues at different scales. If all the eigenvalues are somewhat comparable, then most rank$-k$ approximations are going to be equally good. However, in the presence of eigenvalues at different scales, the quantity $\tr(A^3)$ is going to undergo quite a bit of growth when compared with the Frobenius norm; this leads to a nonlinear feedback loop that leads to a dramatic shrinking of the Frobenius and thus the revealing of good low rank approximations.
   
\subsection{Proof of Corollary 1.}
   \begin{proof}
   Using Lemma 1 with $p_i = 1$ and $p_j = 0$ for $j \neq i$, we have 
   $$ \|B\|_F^2 = \|A\|_F^2 +  \left(   \frac{\|A_i\|^4}{A_{ii}^2} - 2\frac{(A^3)_{ii}}{A_{ii}}  \right).$$
The Corollary follows from applying Lemma 2.
   \end{proof}

\subsection{Proof of Corollary 2.}
\begin{proof} 
The result follows at once from the inequality
$$\max_{1 \leq j \leq n \atop A_{jj} \neq 0}  \frac{\|A_j\|_{\ell^2}^4}{A_{jj}^2} \geq \frac{1}{n} \|A\|_F^2.$$
We first observe that, since $A$ is spd, we have
$$ |A_{ij}| \leq |A_{ii}|^{1/2} \cdot |A_{jj}|^{1/2}.$$
In particular, if a diagonal entry vanishes, $A_{ii} = 0$, then the entire row/column vanishes as well. Therefore, one can write
\begin{align*}
 \|A\|_F^4 &= \left( \sum_{i=1}^{n} \|A_i\|_{\ell^2}^2 \right)^2 =   \left( \sum_{i=1 \atop A_{ii} \neq 0}^{n} \|A_i\|_{\ell^2}^2 \right)^2 = \left( \sum_{i=1  \atop A_{ii} \neq 0 }^{n} \frac{\|A_i\|_{\ell^2}^2}{A_{ii}} A_{ii} \right)^2 \\
 &\leq \left( \sum_{i=1  \atop A_{ii} \neq 0}^{n} \frac{\|A_i\|_{\ell^2}^4}{A_{ii}^2} \right) \left(\sum_{i=1}^{n} A_{ii}^2 \right) \leq \|A\|_F^2  \sum_{i=1  \atop A_{ii} \neq 0}^{n} \frac{\|A_i\|_{\ell^2}^4}{A_{ii}^2}\\
  &\leq \|A\|_F^2 \cdot  n \cdot \max_{1 \leq i \leq n  \atop A_{ii} \neq 0}^{} \frac{\|A_i\|_{\ell^2}^4}{A_{ii}^2}.
\end{align*}
\end{proof}

\end{document}